\documentclass[11pt]{article}
\usepackage{amsmath,amssymb}

\newtheorem{propo}{{\bf Proposition}}[section]
\newtheorem{coro}[propo]{{\bf Corollary}}
\newtheorem{lemma}[propo]{{\bf Lemma}} \newtheorem{theor}[propo]{{\bf
Theorem}} \newtheorem{ex}{{\sc Example}}[section]

\newenvironment{proof}{{\bf Proof.}}{$\Box$}

\def\C{{\mathbb C}}
\def\N{{\mathbb N}}
\begin{document}
\vspace*{1.0in}

\begin{center} ABELIAN SUBALGEBRAS AND IDEALS OF MAXIMAL DIMENSION IN ZINBIEL ALGEBRAS

\end{center}
\bigskip

\centerline {Manuel Ceballos } \centerline {Dpto. de Ingenier\'{\i}a, Universidad Loyola Andaluc\'{\i}a} \centerline
{Av. de las Universidades, s/n, 41704 Dos Hermanas, Sevilla, Spain.}
\centerline {and}

\centerline {David A. Towers} \centerline {Department of
Mathematics, Lancaster University} \centerline {Lancaster LA1 4YF,
England}
\bigskip

\begin{abstract}
In this paper, we compare the abelian subalgebras and ideals of maximal dimension for finite-dimensional Zinbiel algebras. We study Zinbiel algebras containing maximal abelian subalgebras of codimension $1$ and supersolvable Zinbiel algebras in which such subalgebras have codimension $2$, and we also analyze the case of filiform Zinbiel algebras. We give examples to clarify some results, including listing the values for $\alpha$ and $\beta$ for the low dimensional Zinbiel algebras over the complex field that have been classified.

\end{abstract}


\section{Introduction}
\medskip

Zinbiel algebras were introduced by J.-L. Loday \cite{Loday95} in 1995. They are the Koszul dual of Leibniz algebras and J.M. Lemaire (see \cite{Lod01}) proposed the name of Zinbiel for being obtained by writing Leibniz backwards. Leibniz algebras were defined by Loday in 1993 (see \cite{Loday93}). They are a particular case of non-associative algebras and a non-anticommutative generalization of Lie algebras. In fact, they inherit an important property of Lie algebras: the right-multiplication operator is a derivation. Many well-known results on Lie algebras can be extended to Leibniz algebras. In some papers, like \cite{LP,O2005}, the authors study the cohomological and structural properties of Leibniz algebras. Ginzburg and Kapranov introduced and analysed the concept of Koszul dual operads \cite{GK}. Starting from this concept, it was proved in \cite{Loday95} that the dual of the category of Leibniz algebras is defined by the category determined by the so-called Zinbiel identity:
\[  [[x,y],z]=[x,[y,z]]+[x,[z,y]]
\]
Some properties of Zinbiel algebras were studied in \cite{AOK,D,DT}. More concretely, in \cite{DT} the authors proved that every finite-dimensional Zinbiel algebra over an algebraically closed field is solvable and it is nilpotent over the complex number field. Filiform Zinbiel algebras were described and classified in \cite{AOK,CKKK,CCGO}. The classification of complex Zinbiel algebras up to dimension $4$ was obtained in \cite{DT} and \cite{O2002}. Finally, a partial classification of the $5$-dimensional case was done in \cite{ACK}.

We shall call a Zinbiel algebra $Z$ {\em supersolvable} if there is a chain $0=Z_0 \subset Z_1 \subset \ldots \subset Z_{n-1} \subset Z_n=Z$, where $Z_i$ is an $i$-dimensional ideal of $Z$.  We define the following series:
\[ Z^1=Z,Z^{k+1}=[Z,Z^k] \hbox{ and } Z^{(1)}=Z,Z^{(k+1)}=[Z^{(k)},Z^{(k)}] \hbox{ for all } k=2,3, \ldots
\]
We will say that an element of a Zinbiel algebra is {\em left normed} if it is of the form $[a_1,[a_2,[\ldots [a_{n-1},a_n]\ldots]]]$. Then we have the following lemma.
\begin{lemma}\label{nilp} Every element of a Zinbiel algebra that is the product of $n$ elements can be expressed as a linear combination of the $n$ elements with each term being left normed.
\end{lemma}
\begin{proof} We use induction on $n$. If $n=3$ the Zinbiel identity gives that $[[a,b],c]=[a,[b,c]]+[a,[c,b]]$. So assume the result holds for products with less than $n$ elements (where $n>3$). Any product with $n$ elements is of the form $[r,s]$ where $r$ contains $i<n$ elements and $s$ contains $j=n-i$ elements. By the inductive hypothesis, $r$ is a linear combination of left normed elements, each of which has the form $[a,t]$, where $a$ is a single element and $t$ is left normed with $j-1$ elements. Then $[r,s]=[[a,t],s]=[a,[t,s]]+[a,[s,t]]$. But $[t,s]$ and $[s,t]$ each have $n-1$ elements, and so, by the inductive hypothesis, can be written as a linear combination of left normed elements. This completes the induction step.
\end{proof}
\medskip

Then we define $Z$ to be {\em nilpotent} (resp. {\em solvable}) if $Z^n=0$ (resp. $ Z^{(n)}=0$) for some $n \in \N$. It follows from Lemma \ref{nilp} that, in a nilpotent Zinbiel algebra, every product of $n$ elements is zero. The {\em nilradical}, $N(Z)$, (resp. {\em radical}, $R(Z)$) is the largest nilpotent (resp. solvable) ideal of $Z$. We will denote the {\it centre} of $Z$ by $Cen(Z)=\{x \in Z: [x,y]=[y,x]=0, \, \forall \, y \in Z\}$.
\par
A nilpotent Leibniz algebra $Z$ of dimension $n$ is said to be $p$-{\em filiform} if
${\rm dim}(Z^i)=n-p-i+1$, for $2 \leq i \leq n-p+1$. In case that $p=0$ or $p=1$, $Z$ is called null-filiform or filiform, respectively. We say that $Z$ is $k$-{\em abelian} if $k$ is the smallest positive integer such that $Z^k$ is abelian.

Any assumptions on the field $F$ will be specified in each result. We consider the following invariants of $Z$:
$$\alpha(Z) = \max \{\, \dim (A) \, | \, A \,\, {\rm is \,\, an \,\, abelian \,\, subalgebra \,\,
of \,\, } Z\},$$ $$\beta(Z)  = \max \{\,
\dim (B) \, | \, B \,\, {\rm is \,\, an \,\,
abelian \,\, ideal \,\,of \,\,}Z\}.$$
Both invariants are important for many reasons. For example, they are
very useful for the study of contractions
and degenerations; in particular, an algebra $Z_1$ does not degenerate into $Z_2$ if $\dim \alpha (Z_1)> \dim \alpha(Z_2)$ (see \cite[Corollary (6)]{R}).
 There is a large literature, in particular, for low-dimensional Lie algebras, see \cite{GRH,BU10,NPO,SEE,GOR}.

The authors of this paper have already studied these invariants for Lie and Leibniz algebras in
\cite{bc,CT,Tow,CT2}. More concretely, in \cite{bc} it was shown that for a solvable Lie algebra $L$ over an algebraically closed field of characteristic zero, $\alpha(L)=\beta(L)$ and the cases of abelian subalgebras of codimension $1$ and $2$ were also studied. In \cite{CT}, the authors proved that $n$-dimensional supersolvable Lie algebras $L$ with $\alpha(L)=n-2$ have $\beta(L)=n-2$. They also proved the same for nilpotent Lie algebras over a field of characteristic different from two having abelian subalgebras of codimension $3$. This result was generalised for supersolvable Lie algebras in \cite{Tow} and the same was proved for nilpotent Lie algebras containing abelian subalgebras of codimension $4$. In \cite{CT2}, the authors extend their study of the $\alpha$ and $\beta$ invariants to Leibniz algebras. In particular, they proved that $n$-dimensional solvable Leibniz algebras $L$ over a field of characteristic different from two with $\alpha(L)=n-1$ also satisfy $\beta(L)=n-1$. They showed that the same was true for supersolvable Leibniz algebras over any field. This equality of $\alpha$ and $\beta$ was generalised for nilpotent Leibniz algebras over a field of characteristic different from two. They also studied solvable and supersolvable Leibniz algebras containing abelian subalgebras of codimension two. Finally, they proved that there is a unique abelian ideal of maximal dimension for $k$-abelian $p$-filiform Leibniz algebras.

However, we have not found a similar study in the literature for Zinbiel algebras. Therefore, studying abelian subalgebras and ideals of maximal dimension in Zinbiel algebras constitutes the main goal of this paper. The structure of this current paper is as follows. In Section $2$, we study abelian subalgebras of codimension one, proving that every such subalgebra is, in fact, an ideal. Section $3$ is devoted to analysing the case of codimension $2$. Regarding this, we prove that every $n$-dimensional supersolvable Zinbiel algebra $L$ with $\alpha(L)=n-2$ satisfies that $\beta(L)=n-2$ or $\beta(L)=n-3$ and give examples to show that both possibilities can occur.
In Section $4$, we study the case of filiform Zinbiel algebras, proving that there is a unique abelian ideal of maximal dimension. Finally, in Section $5$ we give several tables with the value of alpha and beta invariants for complex Zinbiel algebras of dimension less than or equal to five.

\section{Abelian subalgebras of codimension one}
Here we have the following general result, which is valid over any field.

\begin{theor}\label{cod1} Let $Z$ be a Zinbiel algebra and let $A$ be an abelian subalgebra of codimension one in $Z$. Then $A$ is an ideal of $Z$.
\end{theor}

\begin{proof} Let $Z=A+Fz$.
\par

Suppose first that there is an $a\in A$ such that $[z,a]\notin A$. Then $Z=A+F[z,a]$, so $z=a_1+\lambda[z,a]$ for some $a_1\in A$ and $\lambda \in F$. But now $$[z,a]=\lambda[[z,a],a]=\lambda [z,[a,a]+[a,a]]=0\in A,$$ a contradiction. Hence $[Z,A]\subseteq A$.
\par

Now suppose that there is an $a\in A$ such that $[a,z]\notin A$.  Then $Z=A+F[a,z]$, so $z=a_1+\lambda[a,z]$ for some $a_1\in A$ and $\lambda \in F$. But now $$[a,z]=\lambda[a,[a,z]]=\lambda ([[a,a],z]-[a,[z,a]])\in [A,[Z,A]]=0,$$ a contradiction. Hence $[A,Z]\subseteq A$ and $A$ is an ideal of $Z$.
\end{proof}

\begin{coro}\label{solvcodim1} Let $Z$ be an $n$-dimensional Zinbiel algebra satisfying $\alpha(Z)=n-1$. Then $\beta(Z)=n-1$.
\end{coro}

\section{Abelian subalgebras of codimension two}
First we will need the following Lemma, which is well known for Lie and Leibniz algebras. Its converse will be considered in a later paper, as it is not needed here.

\begin{lemma} Let $Z$ be a supersolvable Zinbiel algebra. Then every maximal subalgebra has codimension one in $Z$.
\end{lemma}
\begin{proof} We use induction on the dimension of $Z$. The result clearly holds if $\dim Z=1$. So suppose it holds for supersolvable Zinbiel algebras of dimension less than $n$ ($n>1$), and let $Z$ have dimension $n$. Let $A$ be a minimal ideal of $Z$ and let $M$ be any maximal subalgebra of $Z$. Then $Z/A$ is supersolvable, so, if $A\subseteq M$, we have that $M$ has codimension one in $Z$, by the inductive hypothesis. If $A\not \subseteq M$, then $Z=M+A$. Moreover, $M\cap A$ is an ideal of $Z$, so $M\cap A=0$ and, again, $M$ has codimension one in $Z$.
\end{proof}

\begin{theor}\label{cod2} Let $Z$ be an $n$-dimensional supersolvable Zinbiel algebra and let $A$ be an maximal abelian subalgebra of codimension two in $Z$, so $\alpha(Z)=n-2$. Then $\beta(Z) = n-2$ or $n-3$.
\end{theor}

\begin{proof} Let $M$ be a maximal subalgebra containing $A$ and let $Z =M+Fz$, where $M=A+Fy$ and $A$ is an ideal of $M$, by Theorem \ref{cod1}. Then
\begin{align} [Z,A]\subseteq M.
\end{align} For, if not, then there is an $a\in A$ such that $Z=M+F[z,a]$, so $z=m+\lambda [z,a]$ for some $m\in M$. But now $[z,a']=[m,a']+\lambda [[z,a],a']\in A$ for all $a'\in A$, since $[[z,a],a']=[z,[a,a']+[a',a]]=0$, a contradiction.
\par

Similarly, we have
\begin{align}[A,Z]\subseteq M.
\end{align} For, otherwise, there is an $a\in A$ such that $Z=M+F[a,z]$, so $z=m+\lambda [a,z]$ for some $m\in M$. But now $[a',z]=[a',m]+\lambda [a',[a,z]]=-\lambda[a',[z,a]]\in [A,[Z,A]]\subseteq [A,M]\subseteq A$ for all $a'\in A$, a contradiction.
\par

Suppose that $A$ is not an ideal of $Z$, so that $Z^{(1)}\not \subseteq A$. Then there is a $k\geq 1$ such that $Z^{(k)}\not \subseteq A$, but $Z^{(k+1)}\subseteq A$. We claim that we may assume that
\begin{align} M \hbox{ is an ideal of } Z.
\end{align} If $Z=A+L^{(k)}$ then $Z^{(1)}\subseteq M+Z^{(k+1)}=M$ and the claim is proved. If not, then $A\subset A+Z^{(k)}\subset Z$, so $A+Z^{(k)}$ is a maximal subalgebra of $Z$. Put $M=A+Z^{(k)}$. Now there exists $r\geq 1$ such that $M=A+Z^{(r)}$ but $M\neq A+Z^{(r-1)}$, so $Z= A+Z^{(r-1)}$. Then $Z^{(1)}\subseteq M+Z^{(r)}=M$ and $M$ is an ideal of $Z$.
\par

Suppose that $[z,A] \subseteq A$. Then $[A,z]\not \subseteq A$, so there exists an $a\in A$ such that $[a,z]\notin A$ and $M=A+F[a,z]$. Then, for all $a'\in A$, $[a',[a,z]]=-[a',[z,a]]=0$, since $[z,a]\in A$. Hence $[A,M]=0$. Also, $[[a,z],a']=[a,[z,a']+[a',z]]=0$, since $[z,a']\in A$ and $[a',z]\in M$, so $[M,A]=0$. Moreover,
$[[a,z],[a,z]]=[a,[z,[a,z]]+[z,[z,a]]]\in [A,M]=0$, so $M$ is abelian, contradicting the maximality of $A$. Hence
\begin{align} [z,A]\not \subseteq A.
\end{align}
So, put $y=[z,a]$. Then $[y,a']=0$ for all $a'\in A$, $y^2=[[z,a],y]=[z,[a,y]+[y,a]]=[[z,y],a]\in [M,A]=0$, and $[y,z]=[[z,a_1],z]=[z^2,a_1]\in [M,A]=0$ whence
\begin{align} [M,A]=0,  y^2=0 \hbox{ and } [y,z]=0.
\end{align}
If there is an $a'\in A$ such that $[a',z]+[z,a'] \notin A$, then $[a'',[a',z]+[z,a']]=[[a'',a'],z]=0$ for all $a''\in A$. But now $[A,M]=0$ and $M$ is abelian, yielding a contradiction again. Hence
\begin{align}[a',z]+[z,a']\in A \hbox{ for all } a'\in A.
\end{align}
Let $A$ be spanned by $a_1, \ldots, a_n$ where $a_1=a$. Then
\begin{align}
[z,a_j]&=\sum_{i=1}^n \lambda_{ji} a_i + \lambda_j y& \hbox{for } j\geq 2 \\
[a_j,z]&=\sum_{i=1}^n \mu_{ji} a_i - \lambda_j y& \hbox{for } j\geq 1.
\end{align}
Putting $b_j=a_j-\lambda_j a_1$ we have
\begin{align}
[z,b_j]&=\lambda_{j1}a_1+\sum_{i=2}^n \lambda_{ji} b_i  \\
[b_j,z]&=\mu_{j1}a_1+\sum_{i=1}^n \mu_{ji} b_i
\end{align} for $j>1$.
 But now $[[z,b_j],z]=[z,[b_j,z]+[z,b_j]]=[z^2,b_j]=0$ for $j\geq 1$, since $z^2\in M$ and using (5). But $[[z,b_j],z]-\lambda_{j1}[a_1,z]\in A$, so $\lambda_{j1}=0$, since $[a_1,z]\notin A$. Also, $[z,[b_j,z]]=\mu_{j1}[z,a_1]+a'$, where $a'\in A$, for $j>1$. But $[z,[b_j,z]]=[[z,b_j],z]-[z,[z,b_j]]\in A$, so $\mu_{j1}=0$ for $j>1$.
\par

Put $B=Fb_2+\dots + Fb_n$. Then $A=Fa_1+B$ and we have shown that $[z,B]+[B,z]\subseteq B$. Now $[y,a_j-\lambda_ja_1]=0$, by (5), so $[y,b_j]=0$ for $j\geq2$. Let $[b_j,y]=\alpha_1 a_1 +\sum_{i=2}^n\alpha_{ji} b_i$. Then $[z,[b_j,y]]=[[z,b_j],y]-[z,[y,b_j]]\in A$ and $[z,[b_j,y]]=\alpha_1[z,a_1]+b$ for some $b\in B$. It follows that $\alpha_1=0$ and $B$ is an abelian ideal of $Z$ of dimension $n-3$.
\end{proof}
\medskip

Note that both possibilities in Theorem \ref{cod2} can occur. The examples in section $5$ show that the first possibility can occur, and the following example shows that the second is also possible.
\begin{ex} Let $Z$ be the six-dimensional Zinbiel algebra with basis $$e_1,e_2,e_3,e_4,e_5.e_6$$ over the complex field $\mathbb{C}$, and non-zero products
\begin{align*} e_1^2&=e_5-e_6,\\
[e_1,e_2]&=e_4+e_5-e_6=[e_3,e_2], \\
[e_1,e_3]&=e_2=[e_1,e_5]=[e_1,e_6]= - [e_3,e_1],\\
[e_5,e_2]&=e_4=[e_6,e_2],\\
[e_5,e_1]&= - e_2+2e_4= [e_6,e_1].
\end{align*}
This is a nilpotent algebra, as $Z^2 = \mathbb{C}e_2+\mathbb{C}e_4+\mathbb{C}(e_5-e_6)$, $Z^3=\mathbb{C}e_4+\mathbb{C}(e_5-e_6)=Cent(Z)$, $Z^4=0$. Moreover, $\mathbb{C}e_3+\mathbb{C}e_4+\mathbb{C}e_5+\mathbb{C}e_6$ is an abelian subalgebra of $Z$ of dimension $4$. We prove that there is no abelian ideal of $Z$ with dimension $4$.

Let us denote by $A$ an abelian ideal of $Z$ of maximal dimension. Then $Cent(Z)\subseteq A$. Consider $x=\sum_{i=1}^6 \lambda_i e_i \in A\setminus Cent(Z)$ ($\lambda_i\in \mathbb{C}$). We can assume that $\lambda_4=0$. Then $[x,e_5] =\lambda_1 e_2$, so either $e_2\in A$ or $\lambda_1=0$. Now, $x^2=\lambda_2\lambda_3(e_4+e_5-e_6)+\lambda_2\lambda_5e_4+\lambda_2\lambda_6e_4$, so $\lambda_2=0$ or $\lambda_3=0$ and $\lambda_5+\lambda_6=0$. It follows that, if $e_2\in A$ then $A=Z^2$.
\par

Suppose that $e_2\notin A$, so $\lambda_1=\lambda_2=0$. Then $[x,e_1]=-\lambda_3e_2+(\lambda_5+\lambda_6)(-e_2+2e_4)$, whence $\lambda_3+\lambda_5+\lambda_6=0$ and $x=(-\lambda_5-\lambda_6)e_3+\lambda_5e_5+\lambda_6e_6=\lambda_5(e_5-e_3)+\lambda_6(e_6-e_3)$. Hence, two further possibilities for $A$ are $Cent(Z) +\mathbb{C}(e_5-e_3)$ and $Cent(Z)+\mathbb{C}(e_6-e_3)$. We cannot extend either any further since $e_5-e_3$, $e_6-e_3$ and $e_5-e_6$ are linearly dependent.

\end{ex}

\section{Abelian subalgebras in filiform Zinbiel algebras}

Throughout this section, we will use the following notation for combinatorial numbers:
\[
C_{i+j-1}^{j}=\binom{i+j-1}{j}=\frac{(i+j-1)!}{j! \, (i-1)!}
\]

\begin{lemma}\label{lemmafiliform}

Let $Z$ be an $n$-dimensional $p$-filiform Zinbiel algebra. Then there exists a basis $e_1,\ldots,e_n$ of $Z$ such that $[e_1,e_i]=e_{i+1}$, $\forall p+1 \leq i \leq n-p-1$.

\end{lemma}

\begin{proof}

Let $Z$ be a $p$-filiform Zinbiel algebra. Then we can choose a basis $e_1,\ldots,e_n$ of $Z$ such that
$$e_1, \ldots, e_{p+1} \in L \setminus L^2 \quad {\rm and} \quad e_{p+i} \in L^{i} \setminus L^{i+1}, \,\, {\rm for} \,\,2 \leq i \leq n-p$$
Since $e_{p+2} \in L^2 \setminus L^3$, $\exists e_i, e_j \in L$ such that $[e_i,e_j]=a_{ij} e_{p+2}$, where $1 \leq i,j \leq n$ and $a_{ij} \neq 0$. We can suppose, without loss of generality, that $i=1, j=p+1$ and $a_{ij}=1$. In this way, we get $[e_{1},e_{p+1}]=e_{p+2}$.
We conclude the proof by following this procedure for every vector $e_{p+i}$, for $i=3, \ldots, n-p-1$ and redefining $e_1$ as a suitable linear combination of $e_1,\ldots,e_{p+1}$.

\end{proof}

\begin{propo}\label{nullfiliform}
Let $Z$ be an $n$-dimensional complex null-filiform Zinbiel algebra. Then $\alpha(Z)=\beta(Z)=n-\lfloor \frac{n}{2} \rfloor$ and there is a unique abelian ideal of maximal dimension.

\end{propo}

\begin{proof}
Let $Z$ be a null-filiform Zinbiel algebra. According to \cite[Theorem 2.4]{AKO}, there is a basis $e_1,\ldots,e_n$ of $Z$ such that
\[
[e_i,e_j]=C_{i+j-1}^{j}\,e_{i+j}, \,\, \forall\, 2 \leq i+j \leq n
\]
This basis also verifies conditions of Lemma \ref{lemmafiliform} and, hence, $[e_1,e_i]=e_{i+1}$, $\forall 1 \leq i \leq n-1$. For this basis, $Z^{k}=span\{e_{k},\ldots,e_n\}$ for $2 \leq k \leq n$.
Let us prove that, $Z^{\lfloor \frac{n}{2} \rfloor +1}=span\{e_{\lfloor \frac{n}{2} \rfloor +1},\ldots,e_n\}$ is the unique abelian ideal of maximal dimension. First, $Z^{\lfloor \frac{n}{2} \rfloor +1}$ is clearly an abelian ideal and $Z^{\lfloor \frac{n}{2} \rfloor}$ is not since $[e_{\lfloor \frac{n}{2} \rfloor },e_{\lfloor \frac{n}{2} \rfloor }] \neq 0$. Let us assume that there is another abelian ideal $A$ which is not contained in $Z^{\lfloor \frac{n}{2} \rfloor +1}$. There is an $e=\displaystyle\sum_{i=1}^n \alpha_i e_i \in A$ but $e \notin Z^{\lfloor \frac{n}{2} \rfloor +1}$. Consequently, there exists $ j$ with $1 \leq j \leq \lfloor \frac{n}{2} \rfloor$ and $\alpha_j \neq 0$. We choose the minimal $j$ satisfying that condition. Then,
$$R_{e} (e_1)=[e_1,e]=\alpha_j e_{j+1} + \sum_{i=j+1}^{n-1} \alpha_i e_{i+1} \in A$$
In fact,
$$R^{\ell}_{e} (e_1)=\alpha_j e_{j+\ell} + \sum_{i=j+1}^{n-\ell} \alpha_i e_{i+\ell} \in A, \quad \forall \, 0 \leq \ell \leq n-j$$
It follows that $\alpha_j e_{j+\ell} \in A$, for $0 \leq \ell \leq n-j$. Since $j \leq \lfloor \frac{n}{2} \rfloor$ and $\alpha_j e_j \in A$, we conclude that $Z^{\lfloor \frac{n}{2} \rfloor} \subset A$ and, therefore, $A$ is not abelian, which is a contradiction.
\end{proof}

\begin{propo}

Let $Z$ be an $n$-dimensional complex filiform Zinbiel algebra. Then there is a unique abelian ideal of maximal dimension and either $\alpha(Z)=\beta(Z)=n-\lfloor \frac{n+1}{2} \rfloor$ or $\alpha(Z)=\beta(Z)=n-\lfloor \frac{n+1}{2}+1 \rfloor$.

\end{propo}

\begin{proof}

Let $Z$ be an $n$-dimensional complex filiform Zinbiel algebra. According to \cite{AKO}[Theorem 2.2], $Z$ is isomorphic to one of the following non-isomorphic algebras:
$$\mathcal{F}_n^1= {\rm span} (e_1, \ldots, e_n): \,\, [e_i,e_j]=C_{i+j-1}^j\, e_{i+j}, \,\, 2 \leq i+j \leq n-1$$
$$\mathcal{F}_n^2= {\rm span} (e_1, \ldots, e_n): \,\, [e_i,e_j]=C_{i+j-1}^j\, e_{i+j}, \,\, 2 \leq i+j \leq n-1; \,\, [e_n,e_1]=e_{n-1}$$
$$\mathcal{F}_n^3= {\rm span} (e_1, \ldots, e_n): \,\, [e_i,e_j]=C_{i+j-1}^j\, e_{i+j}, \,\, 2 \leq i+j \leq n-1; \,\, [e_n,e_n]=e_{n-1}$$
It can be proved in a similar way as in Proposition \ref{nullfiliform} that $A=span\{e_{\lfloor \frac{n+1}{2} \rfloor},\ldots,e_n\}$ is the unique abelian ideal of maximal dimension for $\mathcal{F}_n^1$ and $\mathcal{F}_n^2$. The same happens for $B=span\{e_{\lfloor \frac{n+1}{2} \rfloor},\ldots,e_{n-1}\}$ and Zinbiel algebras $\mathcal{F}_n^3$.
\end{proof}

\section{Tables of $\alpha$ and $\beta$ for Zinbiel algebras}

In Table \ref{alphaybetazinbieldimlessthanfive}, we have computed the value of alpha and beta invariants for complex non-trivial and non-split Zinbiel algebras of dimension less than five. The classification of those algebras is obtained from the results in \cite{AOK,DT,O2002} and bearing in mind the correction done in \cite{KPPV}.
Finally, Tables \ref{alphaybetazinbieldimfive(I)}-\ref{alphaybetazinbieldimfive(VII)} contain the calculations for the $5$-dimensional case, whose classification can be found in \cite{ACK,AJK}.

\begin{table}[htp] \caption{complex non-trivial and non-split Zinbiel algebras of dimension less than five (I).}
\label{alphaybetazinbieldimlessthanfive}
\begin{center}
\begin{tabular}{|c|c|c|c|}
\hline
$Z$ & Products & $\alpha(Z)$ & $\beta(Z)$ \\
\hline
$Z_2^1$ & $[e_1,e_1]=e_2$ & 1 & 1 \\
\hline
$Z_3^1$ & $[e_1,e_1]=e_2, [e_1,e_2]=\frac{1}{2}e_3, [e_2,e_1]=e_3$ & 2 & 2 \\
\hline
$Z_3^2$ & $[e_1,e_2]=e_3, [e_2,e_1]=-e_3$ & 2 & 2 \\
\hline
$Z_3^3$ & \begin{tabular}{c} $[e_1,e_1]=e_3, [e_1,e_2]=e_3$, \\  $[e_2,e_2]=\alpha e_3$, $\alpha \in \C$ \end{tabular} & \begin{tabular}{c} 2 ($\alpha =0$) \\  1 ($\alpha \neq0$)  \end{tabular} & \begin{tabular}{c} 2 ($\alpha =0$) \\  1 ($\alpha \neq0$)  \end{tabular} \\
\hline
$Z_3^4$ & $[e_1,e_1]=e_3, [e_1,e_2]=e_3, [e_2,e_1]=e_3$ & 2 & 2 \\
\hline
$Z_4^1$ & \begin{tabular}{c} $[e_1,e_1]=e_2, [e_1,e_2]=e_3, [e_2,e_1]=2e_3$, \\
$[e_1,e_3]=e_4, [e_2,e_2]=3e_4, [e_3,e_1]=3e_4$ \end{tabular}
 & 2 & 2 \\
\hline
$Z_4^2$ & \begin{tabular}{c} $[e_1,e_1]=e_3, [e_1,e_2]=e_4$, \\
$[e_1,e_3]=e_4, [e_3,e_1]=2e_4$ \end{tabular}
 & 3 & 3 \\
\hline
$Z_4^3$ & \begin{tabular}{c} $[e_1,e_1]=e_3, [e_1,e_3]=e_4$, \\
$[e_2,e_2]=e_4, [e_3,e_1]=2e_4$ \end{tabular}
 & 2 & 2 \\
\hline
$Z_4^4$ &  $[e_1,e_2]=e_3, [e_1,e_3]=e_4$, $[e_2,e_1]=-e_3$
 & 3 & 3 \\
\hline
$Z_4^5$ & \begin{tabular}{c} $[e_1,e_2]=e_3, [e_1,e_3]=e_4$, \\
$[e_2,e_1]=-e_3, [e_2,e_2]=e_4$ \end{tabular}
 & 2 & 2 \\
\hline
$Z_4^6$ & \begin{tabular}{c} $[e_1,e_1]=e_4, [e_1,e_2]=e_3$, \\
$[e_2,e_1]=-e_3, [e_2,e_2]=-2e_3+e_4$ \end{tabular}
 & 2 & 2 \\
\hline
$Z_4^7$ &  $[e_1,e_2]=e_3, [e_2,e_1]=e_4$, $[e_2,e_2]=-e_3$
 & 3 & 3 \\
\hline
$Z_4^{8}(\alpha)$ & \begin{tabular}{c} $[e_1,e_1]=e_3, [e_1,e_2]=e_4$, \\
$[e_2,e_1]=-\alpha e_3, [e_2,e_2]=-e_4$ \end{tabular}
 & \begin{tabular}{c} 3 ($\alpha =1$) \\  2 ($\alpha \neq 1$)   \end{tabular} & \begin{tabular}{c} 3 ($\alpha =1$) \\  2 ($\alpha \neq 1$)   \end{tabular} \\
\hline
$Z_4^{9}(\alpha)$ & \begin{tabular}{c} $[e_1,e_1]=e_4, [e_1,e_2]=\alpha e_4$, \\
$[e_2,e_1]=-\alpha e_4, [e_2,e_2]=e_4, [e_3,e_3]=e_4$ \end{tabular}
 & 1 & 1 \\
\hline
$Z_4^{10}$ & \begin{tabular}{c} $[e_1,e_2]=e_4, [e_1,e_3]=e_4$, \\
$[e_2,e_1]=-e_4, [e_2,e_2]=e_4, [e_3,e_1]=e_4$ \end{tabular}
 & 2 & 2 \\
\hline
$Z_4^{11}$ & \begin{tabular}{c} $[e_1,e_1]=e_4, [e_1,e_2]=e_4$, \\
$[e_2,e_1]=-e_4, [e_3,e_3]=e_4$ \end{tabular}
 & 2 & 2 \\
\hline
$Z_4^{12}$ & $[e_1,e_2]=e_3, [e_2,e_1]=e_4$
 & 3 & 3 \\
\hline
$Z_4^{13}$ & $[e_1,e_2]=e_3, [e_2,e_1]=-e_3, [e_2,e_2]=e_4$
 & 3 & 3 \\
\hline
$Z_4^{14}$ & $[e_2,e_1]=e_4, [e_3,e_1]=e_4$
 & 3 & 3 \\
\hline
$Z_4^{15}(\alpha)$ & \begin{tabular}{c} $[e_1,e_2]=e_4, [e_2,e_2]= e_3$, \\
$[e_2,e_1]=\frac{1+\alpha}{1-\alpha} e_4$, $\alpha \neq 1$ \end{tabular}
 & 3 & 3 \\
\hline
$Z_4^{16}$ & $[e_1,e_2]=e_4, [e_2,e_1]=-e_4, [e_3,e_3]=e_4$
 & 2 & 2 \\
\hline
\end{tabular}
\end{center}
\end{table}

\begin{table}[htp] \caption{$5$-dimensional  non-trivial and non-split complex Zinbiel algebra with $2$-dimensional annihilator (I).}
\label{alphaybetazinbieldimfive(I)}
\begin{center}
\begin{tabular}{|c|c|c|c|}
\hline
$Z$ & Products & $\alpha(Z)$ & $\beta(Z)$ \\
\hline
$Z_5^{1}$ & \begin{tabular}{c} $[e_1,e_1]=e_2, [e_1,e_2]=e_4$, \\ $[e_2,e_1]=2e_4$, $[e_3,e_3]=e_4$ \end{tabular} & 3 & 3 \\
\hline
$Z_5^{2}$ & \begin{tabular}{c} $[e_1,e_1]=e_2, [e_1,e_2]=e_4$, \\ $[e_1,e_3]=e_4$, $[e_2,e_1]=2e_4$ \end{tabular} & 3 & 3 \\
\hline
$Z_5^{3}$ & \begin{tabular}{c} $[e_1,e_1]=e_2, [e_1,e_2]=e_4$, \\ $[e_1,e_3]=e_5$, $[e_2,e_1]=2e_4$ \end{tabular} & 4 & 4 \\
\hline
$Z_5^{4}(\alpha)$ & \begin{tabular}{c} $[e_1,e_1]=e_2, [e_1,e_2]=e_4$, \\ $[e_1,e_3]=\alpha e_5$, $[e_2,e_1]=2e_4$,$[e_3,e_1]=e_5$ \end{tabular} & 4 & 4 \\
\hline
$Z_5^{5}$ & \begin{tabular}{c} $[e_1,e_1]=e_2, [e_1,e_2]=e_4$, \\ $[e_1,e_3]=e_5$, $[e_2,e_1]=2e_4$,$[e_3,e_3]=e_5$ \end{tabular} & 3 & 3 \\
\hline
$Z_5^{6}$ & \begin{tabular}{c} $[e_1,e_1]=e_2, [e_1,e_2]=e_4$, \\ $[e_3,e_3]=e_5$, $[e_2,e_1]=2e_4$ \end{tabular} & 3 & 3 \\
\hline
$Z_5^{7}$ & \begin{tabular}{c} $[e_1,e_1]=e_2, [e_1,e_2]=e_4$, \\ $[e_1,e_3]=e_4$, $[e_2,e_1]=2e_4$,$[e_3,e_3]=e_5$ \end{tabular} & 3 & 3 \\
\hline
$Z_5^{8}$ & \begin{tabular}{c} $[e_1,e_1]=e_2, [e_1,e_2]=e_4$, \\ $[e_1,e_3]=e_4+e_5$, $[e_2,e_1]=2e_4$,$[e_3,e_3]=e_5$ \end{tabular} & 3 & 3 \\
\hline
$Z_5^{9}$ & \begin{tabular}{c} $[e_1,e_1]=e_2, [e_1,e_2]=e_4$, \\ $[e_1,e_3]=e_5$, $[e_2,e_1]=2e_4$,$[e_3,e_1]=e_4+2e_5$ \end{tabular} & 4 & 4 \\
\hline
$Z_5^{10}$ & \begin{tabular}{c} $[e_1,e_1]=e_2, [e_1,e_2]=e_4$, \\ $[e_1,e_3]=e_5$, $[e_2,e_1]=2e_4$,$[e_3,e_3]=e_4$ \end{tabular} & 3 & 3 \\
\hline
$Z_5^{11}(\alpha)$ & \begin{tabular}{c} $[e_1,e_1]=e_2, [e_1,e_2]=e_4$,$[e_1,e_3]=\alpha e_5$, \\  $[e_2,e_1]=2e_4$,$[e_3,e_1]=e_5$, $[e_3,e_3]=e_4$ \end{tabular} & 3 & 3 \\
\hline
$Z_5^{12}$ & \begin{tabular}{c} $[e_1,e_2]=e_3, [e_1,e_3]=e_4$, \\ $[e_2,e_1]=-e_3$, $[e_2,e_2]=e_4$ \end{tabular} & 3 & 3 \\
\hline
$Z_5^{13}$ & $[e_1,e_2]=e_3, [e_1,e_3]=e_4$, $[e_2,e_1]=-e_3$ & 4 & 4 \\
\hline
$Z_5^{14}$ & \begin{tabular}{c} $[e_1,e_1]=e_4, [e_1,e_2]=e_3$, \\ $[e_1,e_3]=e_5$, $[e_2,e_1]=-e_3$ \end{tabular} & 4 & 4 \\
\hline
$Z_5^{15}$ & $[e_1,e_2]=e_3+e_4, [e_1,e_3]=e_5$, $[e_2,e_1]=-e_3$ & 4 & 4 \\
\hline
$Z_5^{16}$ & \begin{tabular}{c} $[e_1,e_2]=e_3, [e_2,e_1]=-e_3$, \\ $[e_1,e_3]=e_5$, $[e_2,e_2]=e_4$ \end{tabular} & 3 & 3 \\
\hline
$Z_5^{17}$ & \begin{tabular}{c} $[e_1,e_1]=e_4, [e_1,e_2]=e_3, [e_2,e_1]=-e_3$, \\ $[e_1,e_3]=e_5$, $[e_2,e_2]=e_4$ \end{tabular} & 3 & 3 \\
\hline
$Z_5^{18}$ & \begin{tabular}{c} $[e_1,e_2]=e_3, [e_2,e_1]=-e_3$, \\ $[e_1,e_3]=e_5$, $[e_2,e_3]=e_4$ \end{tabular} & 3 & 3 \\
\hline
\end{tabular}
\end{center}
\end{table}

\begin{table}[htp] \caption{$5$-dimensional non-trivial and non-split complex  Zinbiel algebra with $2$-dimensional annihilator (II).}
\label{alphaybetazinbieldimfive(II)}
\begin{center}
\begin{tabular}{|c|c|c|c|}
\hline
$Z$ & Products & $\alpha(Z)$ & $\beta(Z)$ \\
\hline
$Z_5^{19}$ & \begin{tabular}{c} $[e_1,e_1]=e_4, [e_1,e_2]=e_3, [e_2,e_1]=-e_3$, \\ $[e_1,e_3]=e_5$, $[e_2,e_3]=e_4$ \end{tabular} & 3 & 3 \\
\hline
$Z_5^{20}$ & \begin{tabular}{c} $[e_1,e_2]=e_3+e_4, [e_1,e_3]=e_5$, \\ $[e_2,e_1]=-e_3$,  $[e_2,e_3]=e_4$ \end{tabular} & 3 & 3 \\
\hline
$Z_5^{21}$ & \begin{tabular}{c} $[e_1,e_1]=e_4, [e_1,e_2]=e_3, [e_2,e_1]=-e_3$, \\ $[e_1,e_3]=e_5, [e_2,e_2]=e_4$, $[e_2,e_3]=e_4$ \end{tabular} & 3 & 3 \\
\hline
$Z_5^{22}$ & \begin{tabular}{c} $[e_1,e_1]=e_4, [e_1,e_2]=e_3, [e_2,e_1]=-e_3$, \\ $[e_1,e_3]=e_5$, $[e_2,e_2]=e_5$ \end{tabular} & 3 & 3 \\
\hline
$Z_5^{23}$ & \begin{tabular}{c} $[e_1,e_2]=e_3+e_4, [e_1,e_3]=e_5$, \\ $[e_2,e_1]=-e_3$, $[e_2,e_2]=e_5$ \end{tabular} & 3 & 3 \\
\hline
\end{tabular}
\end{center}
\end{table}

\begin{table}[htp] \caption{$5$-dimensional non-trivial and non-split complex Zinbiel algebras with $1$-dimensional annihilator (I).}
\label{alphaybetazinbieldimfive(III)}
\begin{center}
\begin{tabular}{|c|c|c|c|}
\hline
$Z$ & Products & $\alpha(Z)$ & $\beta(Z)$ \\
\hline
$Z_5^{24}$ & \begin{tabular}{c} $[e_1,e_1]=e_2, [e_1,e_2]=e_5$, \\ $[e_1,e_3]=e_5$, $[e_2,e_1]=2e_5$, $[e_4,e_4]=e_5$ \end{tabular} & 3 & 3 \\
\hline
$Z_5^{25}(\alpha)$ & \begin{tabular}{c} $[e_1,e_1]=e_2, [e_1,e_2]=e_5$, \\ $[e_2,e_1]=2e_5$, $[e_3,e_4]=e_5$, $[e_4,e_3]=\alpha e_5$ \end{tabular} & 3 & 3 \\
\hline
$Z_5^{26}$ & \begin{tabular}{c} $[e_1,e_1]=e_2, [e_1,e_2]=e_5, [e_2,e_1]=2e_5$, \\ $[e_3,e_3]=e_5$, $[e_3,e_4]=e_5$, $[e_4,e_3]=-e_5$ \end{tabular} & 3 & 3 \\
\hline
$Z_5^{27}$ & \begin{tabular}{c} $[e_1,e_1]=e_2, [e_1,e_2]=e_5$, $[e_1,e_4]=e_5$, \\ $[e_2,e_1]=2 e_5$, $[e_3,e_4]=e_5$,$[e_4,e_3]=2e_5$ \end{tabular} & 3 & 3 \\
\hline
$Z_5^{28}$ & \begin{tabular}{c} $[e_1,e_1]=e_3, [e_1,e_3]=e_5$, $[e_2,e_2]=e_4$, \\ $[e_2,e_4]=e_5$, $[e_3,e_1]=2e_5$,$[e_4,e_2]=2e_5$ \end{tabular} & 3 & 3 \\
\hline
$Z_5^{29}$ & \begin{tabular}{c} $[e_1,e_2]=e_3, [e_1,e_3]=e_5$, \\ $[e_2,e_1]=-e_3$, $[e_2,e_4]=e_5$ \end{tabular} & 3 & 3 \\
\hline
$Z_5^{30}$ & \begin{tabular}{c} $[e_1,e_2]=e_3, [e_1,e_3]=e_5$, \\ $[e_2,e_1]=-e_3$, $[e_4,e_1]=e_5$ \end{tabular} & 4 & 4 \\
\hline
$Z_5^{31}$ & \begin{tabular}{c} $[e_1,e_2]=e_3, [e_1,e_3]=e_5$, \\ $[e_2,e_1]=-e_3$, $[e_2,e_2]=e_5$,$[e_4,e_1]=e_5$ \end{tabular} & 3 & 3 \\
\hline
\end{tabular}
\end{center}
\end{table}

\begin{table}[htp] \caption{$5$-dimensional non-trivial and non-split complex Zinbiel algebras with $1$-dimensional annihilator (II).}
\label{alphaybetazinbieldimfive(IV)}
\begin{center}
\begin{tabular}{|c|c|c|c|}
\hline
$Z$ & Products & $\alpha(Z)$ & $\beta(Z)$ \\
\hline
$Z_5^{32}$ & \begin{tabular}{c} $[e_1,e_2]=e_3, [e_1,e_3]=e_5$, \\ $[e_2,e_1]=-e_3$, $[e_2,e_4]=e_5$,$[e_4,e_1]=e_5$ \end{tabular} & 3 & 3 \\
\hline
$Z_5^{33}(\alpha)$ & \begin{tabular}{c} $[e_1,e_2]=e_3, [e_1,e_4]=\alpha e_5$, \\ $[e_2,e_1]=-e_3$, $[e_2,e_3]=e_5$,$[e_4,e_1]=e_5$ \end{tabular} & 3 & 3 \\
\hline
$Z_5^{34}$ & \begin{tabular}{c} $[e_1,e_1]=e_5, [e_1,e_2]=e_3$,$[e_1,e_4]=-e_5$, \\  $[e_2,e_1]=-e_3$,$[e_2,e_3]=e_5$, $[e_4,e_1]=e_5$ \end{tabular} & 3 & 3 \\
\hline
$Z_5^{35}$ & \begin{tabular}{c} $[e_1,e_2]=e_3, [e_2,e_1]=-e_3$, \\ $[e_2,e_3]=e_5$, $[e_2,e_4]=e_5$ \end{tabular} & 4 & 4 \\
\hline
$Z_5^{36}$ &  \begin{tabular}{c} $[e_1,e_1]=e_5, [e_1,e_2]=e_3$, \\ $[e_2,e_1]=-e_3$, $[e_2,e_3]=e_5$, $[e_4,e_4]=e_5$  \end{tabular} & 3 & 3 \\ 
\hline
$Z_5^{37}(\alpha)$ & \begin{tabular}{c} $[e_1,e_1]=\alpha e_5, [e_1,e_2]=e_3$, $[e_1,e_4]=e_5$, \\ $[e_2,e_1]=-e_3$, $[e_2,e_3]=e_5$, $[e_4,e_4]=e_5$ \end{tabular} & 3 & 3 \\
\hline
$Z_5^{38}$ & $[e_1,e_2]=e_3, [e_2,e_1]=-e_3$, $[e_4,e_3]=e_5$ & 3 & 3 \\
\hline
$Z_5^{39}$ & \begin{tabular}{c} $[e_1,e_1]=e_5, [e_1,e_2]=e_3$, \\ $[e_2,e_1]=-e_3$, $[e_4,e_3]=e_5$ \end{tabular} & 3 & 3 \\
\hline
$Z_5^{40}$ & $[e_1,e_2]=e_3+e_5, [e_2,e_1]=-e_3, [e_4,e_3]=e_5$ & 3 & 3 \\
\hline
$Z_5^{41}$ & \begin{tabular}{c} $[e_1,e_2]=e_3, [e_2,e_1]=-e_3$, \\ $[e_2,e_4]=e_5$, $[e_4,e_3]=e_5$ \end{tabular} & 3 & 3 \\
\hline
$Z_5^{42}$ & \begin{tabular}{c} $[e_1,e_1]=e_5, [e_1,e_2]=e_3, [e_2,e_1]=-e_3$, \\ $[e_2,e_4]=e_5$, $[e_4,e_3]=e_5$ \end{tabular} & 3 & 3 \\
\hline
$Z_5^{43}$ & \begin{tabular}{c} $[e_1,e_2]=e_3+e_5, [e_2,e_1]=-e_3$, \\ $[e_2,e_4]=e_5$,  $[e_4,e_3]=e_5$ \end{tabular} & 3 & 3 \\
\hline
$Z_5^{44}$ & \begin{tabular}{c} $[e_1,e_2]=e_3, [e_2,e_1]=-e_3, [e_2,e_2]=e_5$, \\ $[e_2,e_4]=e_5, [e_4,e_3]=e_5$ \end{tabular} & 3  & 3 \\
\hline
$Z_5^{45}$ & \begin{tabular}{c} $[e_1,e_1]=e_5, [e_1,e_2]=e_3, [e_2,e_1]=-e_3$, \\ $[e_2,e_2]=e_5$, $[e_2,e_4]=e_5$, $[e_4,e_3]=e_5$ \end{tabular} & 3 & 3 \\
\hline
$Z_5^{46}$ & \begin{tabular}{c} $[e_1,e_2]=e_3, [e_1,e_3]=e_5$, $[e_1,e_4]=-e_5$, \\ $[e_2,e_1]=e_4$, $[e_2,e_2]=-e_3$, $[e_2,e_3]=-e_5$, \\
 $[e_2,e_4]=e_5, [e_3,e_2]=-2e_5$ \end{tabular} & 3 & 3 \\
\hline
$Z_5^{47}$ & \begin{tabular}{c} $[e_1,e_1]=e_3, [e_1,e_2]=e_4$, $[e_1,e_4]=-e_5$, \\ $[e_2,e_1]=-e_3$, $[e_2,e_2]=-e_4$, $[e_2,e_4]=e_5$, \\
 $[e_3,e_2]=-e_5, [e_4,e_1]=-e_5$, $[e_4,e_2]=2e_5$ \end{tabular} & 3 & 3 \\
\hline
\end{tabular}
\end{center}
\end{table}

\begin{table}[htp] \caption{$5$-dimensional non-trivial and non-split complex Zinbiel algebras with $1$-dimensional annihilator (III).}
\label{alphaybetazinbieldimfive(V)}
\begin{center}
\begin{tabular}{|c|c|c|c|}
\hline
$Z$ & Products & $\alpha(Z)$ & $\beta(Z)$ \\
\hline
$Z_5^{48}$ & \begin{tabular}{c} $[e_1,e_1]=e_3+e_5, [e_1,e_2]=e_4$, $[e_1,e_3]=-e_5$, \\ $[e_1,e_4]=e_5$, $[e_2,e_1]=-e_3$, $[e_2,e_2]=-e_4$, \\
 $[e_2,e_3]=e_5, [e_2,e_4]=-e_5$, $[e_3,e_1]=-2e_5$, \\ $[e_3,e_2]=2e_5$, $[e_4,e_1]=2e_5$, $[e_4,e_2]=-2e_5$ \end{tabular} & 3 & 3 \\
\hline
$Z_5^{49}$ & \begin{tabular}{c} $[e_1,e_1]=e_3, [e_1,e_2]=e_4$, $[e_1,e_3]=-e_5$, \\ $[e_1,e_4]=e_5$, $[e_2,e_1]=-e_3$, $[e_2,e_2]=-e_4$, \\
 $[e_2,e_3]=e_5, [e_2,e_4]=-e_5$, $[e_3,e_1]=-2e_5$, \\ $[e_3,e_2]=2e_5$, $[e_4,e_1]=2e_5$, $[e_4,e_2]=-2e_5$ \end{tabular} & 3 & 3 \\
\hline
$Z_5^{50}$ & \begin{tabular}{c} $[e_1,e_2]=e_3, [e_1,e_3]=-e_5$, $[e_1,e_4]=e_5$, \\ $[e_2,e_1]=e_4$, $[e_2,e_3]=-e_5$, $[e_2,e_4]=e_5$  \end{tabular} & 3 & 3 \\
\hline
$Z_5^{51}$ & \begin{tabular}{c} $[e_1,e_2]=e_3, [e_1,e_3]=-e_5$, \\ $[e_1,e_4]=e_5$, $[e_2,e_1]=e_4$, $[e_2,e_2]=e_5$ \end{tabular} & 3 & 3 \\
\hline
$Z_5^{52}$ & \begin{tabular}{c} $[e_1,e_2]=e_3, [e_1,e_3]=-e_5$, \\ $[e_1,e_4]=e_5$, $[e_2,e_1]=e_4$ \end{tabular} & 4 & 4 \\
\hline
\begin{tabular}{c} $Z_5^{53}(\alpha)$ \\  $\alpha \neq -1$ \end{tabular} & \begin{tabular}{c} $[e_1,e_2]=e_4, [e_1,e_3]=(\alpha+1)e_5$, \\ $[e_2,e_1]=\alpha e_4$,  $[e_2,e_2]=e_3$, $[e_2,e_4]=2\alpha e_5$, \\
$[e_3,e_1]=2\alpha(\alpha+1) e_5$, $[e_4,e_2]=2(\alpha+1)e_5$ \end{tabular} & 3 & 3 \\
\hline
$Z_5^{54}$ & \begin{tabular}{c} $[e_1,e_2]=e_4, [e_1,e_3]=e_5, [e_2,e_1]=e_5$, \\ $[e_2,e_2]=e_3, [e_4,e_2]=2e_5$ \end{tabular} & 3 & 3 \\
\hline
$Z_5^{55}$ & \begin{tabular}{c} $[e_1,e_1]=e_5, [e_1,e_2]=e_4$, $[e_1,e_3]=\frac{1}{2}e_5$, \\ $[e_2,e_1]=-\frac{1}{2}e_4$, $[e_2,e_2]=e_3$, $[e_2,e_4]=-e_5$, \\
 $[e_3,e_1]=-\frac{1}{2}e_5, [e_4,e_2]=e_5$ \end{tabular} & 3 & 3 \\
\hline
$Z_5^{56}$ & \begin{tabular}{c} $[e_1,e_2]=e_4, [e_1,e_3]=\frac{1}{2}e_5$, $[e_2,e_1]=-\frac{1}{2}e_4$, \\ $[e_2,e_2]=e_3$, $[e_2,e_3]=e_5$
 $[e_2,e_4]=-e_5$, \\ $[e_3,e_1]=-\frac{1}{2}e_5$, $[e_3,e_2]=2e_5$, $[e_4,e_2]=e_5$  \end{tabular} & 3 & 3 \\
\hline
$Z_5^{57}$ & \begin{tabular}{c} $[e_1,e_1]=e_5, [e_1,e_2]=e_4, [e_1,e_3]=\frac{1}{2}e_5$, \\ $[e_2,e_1]=-\frac{1}{2}e_4, [e_2,e_2]=e_3$, $[e_2,e_3]=e_5$, \\
$[e_2,e_4]=-e_5$, $[e_3,e_1]=-\frac{1}{2}e_5$, \\ $[e_3,e_2]=2e_5$, $[e_4,e_2]=e_5$ \end{tabular} & 3 & 3 \\
\hline
$Z_5^{58}$ & \begin{tabular}{c} $[e_1,e_2]=e_4, [e_1,e_4]=e_5$, $[e_2,e_1]=-e_4$, \\ $[e_2,e_2]=e_3$, $[e_2,e_3]=e_5$, $[e_3,e_2]=2e_5$ \end{tabular} & 3 & 3 \\
\hline
$Z_5^{59}$ & \begin{tabular}{c} $[e_1,e_1]=e_5, [e_1,e_2]=e_4$, $[e_2,e_1]=-e_4$, \\ $[e_2,e_2]=e_3$, $[e_2,e_3]=e_5$, \\
 $[e_2,e_4]=e_5$, $[e_3,e_2]=2e_5$  \end{tabular} & 3 & 3 \\
\hline
\end{tabular}
\end{center}
\end{table}

\begin{table}[htp] \caption{$5$-dimensional non-trivial and non-split complex Zinbiel algebras with $1$-dimensional annihilator (IV).}
\label{alphaybetazinbieldimfive(VI)}
\begin{center}
\begin{tabular}{|c|c|c|c|}
\hline
$Z$ & Products & $\alpha(Z)$ & $\beta(Z)$ \\
\hline
$Z_5^{60}$ & \begin{tabular}{c} $[e_1,e_2]=e_4, [e_2,e_1]=-e_4, [e_2,e_2]=e_3$, \\ $[e_2,e_3]=e_5, [e_2,e_4]=e_5$, $[e_3,e_2]=2e_5$ \end{tabular} & 4 & 4 \\
\hline
$Z_5^{61}$ & \begin{tabular}{c} $[e_1,e_1]=e_2, [e_1,e_2]=e_3$, $[e_1,e_3]=e_5$, \\ $[e_2,e_1]=2e_3$, $[e_2,e_2]=3e_5$,  \\
$[e_3,e_1]=3e_5$, $[e_4,e_4]=e_5$ \end{tabular} & 3 & 3 \\
\hline
$Z_5^{62}$ & \begin{tabular}{c} $[e_1,e_1]=e_2, [e_1,e_2]=e_3$, $[e_1,e_3]=e_5$, \\ $[e_1,e_4]=e_5$, $[e_2,e_1]=2e_3$,
\\ $[e_2,e_2]=3e_5$, $[e_3,e_1]=3e_5$  \end{tabular} & 3 & 3 \\
\hline
$Z_5^{63}$ & \begin{tabular}{c} $[e_1,e_1]=e_2, [e_1,e_2]=\frac{1}{2}e_3$, $[e_1,e_3]=2e_4$, \\ $[e_1,e_4]=e_5$, $[e_2,e_1]=e_3$, $[e_2,e_2]=3e_4$,
\\  $[e_2,e_3]=8e_5$, $[e_3,e_1]=6e_4$, \\ $[e_3,e_2]=12e_5$, $[e_4,e_1]=4e_5$   \end{tabular} & 3 & 3 \\
\hline
$Z_5^{64}$ & \begin{tabular}{c} $[e_1,e_1]=e_2, [e_1,e_2]=e_4$, \\  $[e_1,e_3]=e_5$, $[e_2,e_1]=2e_4$  \end{tabular} & 4 & 4 \\
\hline
$Z_5^{65}(\alpha)$ & \begin{tabular}{c} $[e_1,e_1]=e_2, [e_1,e_2]=e_4$, \\  $[e_1,e_3]=\alpha e_5$, $[e_2,e_1]=2e_4$, $[e_3,e_1]=e_5$  \end{tabular} & 4 & 4 \\
\hline
$Z_5^{66}$ & \begin{tabular}{c} $[e_1,e_1]=e_2, [e_1,e_2]=e_4$, \\  $[e_1,e_3]=e_5$, $[e_2,e_1]=2e_4$, $[e_3,e_3]=e_5$  \end{tabular} & 3 & 3 \\
\hline
$Z_5^{67}$ & \begin{tabular}{c} $[e_1,e_1]=e_2, [e_1,e_2]=e_4$, \\  $[e_2,e_1]=2e_4$, $[e_3,e_3]=e_5$  \end{tabular} & 3 & 3 \\
\hline
$Z_5^{68}$ & \begin{tabular}{c} $[e_1,e_1]=e_2, [e_1,e_2]=e_4$, \\  $[e_1,e_3]=e_4$, $[e_2,e_1]=2e_4$, $[e_3,e_3]=e_5$  \end{tabular} & 3 & 3 \\
\hline
$Z_5^{69}$ & \begin{tabular}{c} $[e_1,e_1]=e_2, [e_1,e_2]=e_4$, \\  $[e_1,e_3]=e_4+e_5$, $[e_2,e_1]=2e_4$, $[e_3,e_3]=e_5$  \end{tabular} & 3 & 3 \\
\hline
$Z_5^{70}$ & \begin{tabular}{c} $[e_1,e_1]=e_2, [e_1,e_2]=e_4$, \\  $[e_1,e_3]=e_5$, $[e_2,e_1]=2e_4$, $[e_3,e_1]=e_4+2e_5$  \end{tabular} & 4 & 4 \\
\hline
$Z_5^{71}$ & \begin{tabular}{c} $[e_1,e_1]=e_2, [e_1,e_2]=e_4$, \\  $[e_1,e_3]=e_5$, $[e_2,e_1]=2e_4$, $[e_3,e_3]=e_4$  \end{tabular} & 3 & 3 \\
\hline
$Z_5^{72}(\alpha)$ & \begin{tabular}{c} $[e_1,e_1]=e_2, [e_1,e_2]=e_4$, $[e_1,e_3]=\alpha e_5$, \\   $[e_2,e_1]=2e_4$, $[e_3,e_1]=e_5$, $[e_3,e_3]=e_4$  \end{tabular} & 3 & 3 \\
\hline
$Z_5^{73}$ & \begin{tabular}{c} $[e_1,e_1]=e_4, [e_1,e_2]=e_3$, \\  $[e_1,e_3]=e_5$, $[e_2,e_1]=-e_3$  \end{tabular} & 4 & 4 \\
\hline
$Z_5^{74}$ & \begin{tabular}{c} $[e_1,e_2]=e_3+e_4$, $[e_1,e_3]=e_5$, $[e_2,e_1]=-e_3$  \end{tabular} & 4 & 4 \\
\hline
\end{tabular}
\end{center}
\end{table}

\begin{table}[htp] \caption{$5$-dimensional non-trivial and non-split complex Zinbiel algebras with $1$-dimensional annihilator (V).}
\label{alphaybetazinbieldimfive(VII)}
\begin{center}
\begin{tabular}{|c|c|c|c|}
\hline
$Z$ & Products & $\alpha(Z)$ & $\beta(Z)$ \\
\hline
$Z_5^{75}$ & \begin{tabular}{c} $[e_1,e_2]=e_3, [e_1,e_3]=e_5$, \\  $[e_2,e_1]=-e_3$, $[e_2,e_2]=e_4$  \end{tabular} & 3 & 3 \\
\hline
$Z_5^{76}$ & \begin{tabular}{c} $[e_1,e_1]=e_4, [e_1,e_2]=e_3$, \\  $[e_1,e_3]=e_5$, $[e_2,e_1]=-e_3$, $[e_2,e_2]=e_4$  \end{tabular} & 3 & 3 \\
\hline
$Z_5^{77}$ & \begin{tabular}{c} $[e_1,e_2]=e_3, [e_1,e_3]=e_5$, \\  $[e_2,e_1]=-e_3$, $[e_2,e_3]=e_4$  \end{tabular} & 3 & 3 \\
\hline
$Z_5^{78}$ & \begin{tabular}{c} $[e_1,e_1]=e_4, [e_1,e_2]=e_3$, \\  $[e_1,e_3]=e_5$, $[e_2,e_1]=-e_3$, $[e_2,e_3]=e_4$  \end{tabular} & 3 & 3 \\
\hline
$Z_5^{79}$ & \begin{tabular}{c} $[e_1,e_2]=e_3+e_4, [e_1,e_3]=e_5$, \\  $[e_2,e_1]=-e_3$, $[e_2,e_3]=e_4$  \end{tabular} & 3 & 3 \\
\hline
$Z_5^{80}$ & \begin{tabular}{c} $[e_1,e_1]=e_4, [e_1,e_2]=e_3$, $[e_1,e_3]=e_5$ \\  $[e_2,e_1]=-e_3$, $[e_2,e_2]=e_4$, $[e_2,e_3]=e_4$  \end{tabular} & 3 & 3 \\
\hline
$Z_5^{81}$ & \begin{tabular}{c} $[e_1,e_1]=e_4, [e_1,e_2]=e_3$, \\  $[e_1,e_3]=e_5$, $[e_2,e_1]=-e_3$, $[e_2,e_2]=e_5$  \end{tabular} & 3 & 3 \\
\hline
$Z_5^{82}$ & \begin{tabular}{c} $[e_1,e_2]=e_3+e_4, [e_1,e_3]=e_5$, \\  $[e_2,e_1]=-e_3$, $[e_2,e_2]=e_5$  \end{tabular} & 3 & 3 \\
\hline
\end{tabular}
\end{center}
\end{table}

{\bf Acknowledgement} The authors are grateful to Professor Kaygorodov for spotting some typographical errors and for providing reference \cite{AJK}.

\end{document}